\documentclass[12pt]{amsart}

\usepackage{amstext}
\usepackage{amsthm}
\usepackage{amsmath}
\usepackage{amssymb}
\usepackage{latexsym}
\usepackage{amsfonts}
\usepackage{graphicx}
\usepackage{texdraw}
\usepackage{graphpap}
\usepackage{subfigure}

\usepackage[pagebackref,hypertexnames=false, colorlinks, citecolor=red, linkcolor=red]{hyperref} 
\usepackage[backrefs]{amsrefs}

\input txdtools

\DeclareFontFamily{U}{wncy}{}
\DeclareFontShape{U}{wncy}{m}{n}{<->wncyr10}{}
\DeclareSymbolFont{mcy}{U}{wncy}{m}{n}
\DeclareMathSymbol{\Sh}{\mathord}{mcy}{"58}

\begin{document}

\newcommand{\ci}[1]{_{ {}_{\scriptstyle #1}}}

\newcommand{\unit}{1\!\!1}

\newcommand*\rfrac[2]{{}^{#1}\!/_{#2}}

\newcommand{\norm}[1]{\ensuremath{\|#1\|}}
\newcommand{\abs}[1]{\ensuremath{\vert#1\vert}}
\newcommand{\p}{\ensuremath{\partial}}
\newcommand{\pr}{\mathcal{P}}

\newcommand{\pbar}{\ensuremath{\bar{\partial}}}
\newcommand{\db}{\overline\partial}
\newcommand{\D}{\mathbb{D}}
\newcommand{\B}{\mathbb{B}}
\newcommand{\Sp}{\mathbb{S}}
\newcommand{\T}{\mathbb{T}}
\newcommand{\R}{\mathbb{R}}
\newcommand{\Z}{\mathbb{Z}}
\newcommand{\C}{\mathbb{C}}
\newcommand{\N}{\mathbb{N}}
\newcommand{\scrH}{\mathcal{H}}
\newcommand{\scrL}{\mathcal{L}}
\newcommand{\td}{\widetilde\Delta}

\newcommand{\La}{\langle }
\newcommand{\Ra}{\rangle }
\newcommand{\rk}{\operatorname{rk}}
\newcommand{\card}{\operatorname{card}}
\newcommand{\ran}{\operatorname{Ran}}
\newcommand{\osc}{\operatorname{OSC}}
\newcommand{\im}{\operatorname{Im}}
\newcommand{\re}{\operatorname{Re}}
\newcommand{\tr}{\operatorname{tr}}
\newcommand{\vf}{\varphi}
\newcommand{\f}[2]{\ensuremath{\frac{#1}{#2}}}


\newcommand{\entrylabel}[1]{\mbox{#1}\hfill}

\newenvironment{entry}
{\begin{list}{X}%
  {\renewcommand{\makelabel}{\entrylabel}%
      \setlength{\labelwidth}{55pt}%
      \setlength{\leftmargin}{\labelwidth}
      \addtolength{\leftmargin}{\labelsep}%
   }%
}%
{\end{list}}


\numberwithin{equation}{section}

\newtheorem{thm}{Theorem}[section]
\newtheorem{lm}[thm]{Lemma}
\newtheorem{cor}[thm]{Corollary}
\newtheorem{conj}[thm]{Conjecture}
\newtheorem{prob}[thm]{Problem}
\newtheorem{prop}[thm]{Proposition}
\newtheorem*{prop*}{Proposition}

\theoremstyle{remark}
\newtheorem{rem}[thm]{Remark}
\newtheorem*{rem*}{Remark}

\title{Bloom's Inequality: Commutators in a Two-Weight Setting}

\author{Irina Holmes}
\address{Irina Holmes, School of Mathematics\\ Georgia Institute of Technology\\ 686 Cherry Street\\ Atlanta, GA USA 30332-0160}
\email{irina.holmes@math.gatech.edu}

\author[M. T. Lacey]{Michael T. Lacey$^{\dagger}$}
\address{Michael T. Lacey, School of Mathematics\\ Georgia Institute of Technology\\ 686 Cherry Street\\ Atlanta, GA USA 30332-0160}
\email{lacey@math.gatech.edu}
\thanks{$\dagger$  Research supported in part by a National Science Foundation DMS grant \#1265570.}

\author[B. D. Wick]{Brett D. Wick$^{\ddagger}$}
\address{Brett D. Wick, Department of Mathematics\\ Washington University -- Saint Louis\\ One Brookings Drive\\Saint Louis, MO USA 63130-4899}
\email{wick@math.wustl.edu}
\thanks{$\ddagger$  Research supported in part by a National Science Foundation DMS grant \#0955432 and \#1560955.}

\subjclass[2000]{Primary }

\begin{abstract} 
In 1985, Bloom  characterized the boundedness of the commutator  $[b,H]$ as a map 
between a pair of weighted $ L ^{p}$ spaces, where both weights are in $ A_p$.  The characterization is 
in terms of a novel $ BMO$ condition.   
We give a `modern' proof of this result, in the case of $ p=2$.  In a subsequent paper, this argument will 
be used to generalize Bloom's result to  all Calder\'on-Zygmund operators and dimensions. 
\end{abstract}

\maketitle
\setcounter{tocdepth}{1}
\tableofcontents

\section{Introduction and Statement of Main Results}

Let $\mu$ be a weight on $\mathbb{R}$, i.e.~a function that is positive almost everywhere and is locally integrable.  Then define $L^2(\mathbb{R};\mu)\equiv L^2(\mu)$ to be the space of functions which are square integrable with respect to the measure $\mu(x) dx$, namely
$$
\left\Vert f\right\Vert_{L^2(\mu)}^2\equiv\int_{\mathbb{R}} \left\vert f(x)\right\vert^2 \mu(x)dx.
$$
For an interval $I$, let $\left\langle \mu\right\rangle_{I}\equiv \frac{1}{\vert I\vert}\int_{I} \mu(x)dx$.  And, similarly, set $\mathbb{E}_I^{\mu}(g)\equiv \frac{1}{\mu(I)}\int_{I} g\mu dx$.

In \cite{Bloom} Bloom considers the behavior of the commutator
\begin{equation*}
[b, H] : L ^{p} (\lambda)  \mapsto L ^{p} (\mu)
\end{equation*}
where $H$ is the Hilbert transform.  When the weights $\mu=\lambda\in A_2$ then it is well-known that  boundedness  is characterized by   $b\in BMO$.  Bloom  however works in the setting of $ \mu \neq \lambda \in A_2$, finding a characterization in terms of a $BMO$ space adapted to the weight $\rho=\left(\frac{\mu}{\lambda}\right)^{\frac{1}{p}}$, 
namely 
$$  \label{e:bmo-rho}
\left\Vert b\right\Vert_{BMO_{\rho}}\equiv\sup_{I} \left(\frac{1}{\rho(I)}\int_{I} \left\vert b(x)-\left\langle b\right\rangle_I\right\vert^2 dx\right)^{\frac{1}{2}}.
$$
Recall that $ \lambda \in A_p$ if and only if the supremum over intervals below is finite. 
\begin{equation*}
[ \lambda ] _{A_p} = \sup _{I} \langle \lambda  \rangle_I \langle  \lambda ^{1-p'} \rangle ^{p-1} < \infty . 
\end{equation*}

\begin{thm}[Bloom, \cite{Bloom}*{Theorem 4.2}]
\label{t:bloom} 
Let $ 1< p < \infty $, $ \mu ,\lambda \in A_p$. Set $ \rho = \left(\frac{\mu}{\lambda}\right) ^{\frac{1}{p}}$. Then, 
$$
\left\Vert [b,H]:L^p(\mu)\to L^p(\lambda)\right\Vert\approx \left\Vert b\right\Vert_{BMO_{\rho}}.
$$ 
\end{thm}

The space $BMO_{\rho}=BMO$ when $\mu=\lambda$, and this case is well-known.
But, the general case is rather delicate, as there are  three independent objects in the commutator, 
the two weights and the symbol $ b$.  It is remarkable that there is a single condition involving all three which characterizes the boundedness of the commutator.  

Commutator estimates are interesting in that operator bounds are characterized in terms of function classes. 
They generalize Hankel operators, encode weak-factorization results for the Hardy space, and can be used 
to derive div-curl estimates.  Bloom himself applied his inequality to matrix weights. 
As far as we know, many of these topics remain unexplored in the setting of Bloom's 
inequality, and we hope to return to these topics in  future papers.  

Weighted estimates for commutators are complicated, since $ [b,H]$ is essentially the composition of 
$ H$ with paraproduct operators, see \eqref{e:expand} below.  This makes two weight estimates for 
commutators very difficult.  But, the key assumption of both weights being in $ A_2$ allows several 
proof strategies that are not available in the general two weight case.  A key property is the `joint $ A _{\infty }$ property,' namely that  one can quantitatively control Carleson sequences of intervals in both measures.  Bloom's argument is based upon interesting sharp function inequality for the upper bound, and involves an \emph{ad hoc} argument in the lower bound.  

We give an alternate proof of Theorem \ref{t:bloom} in the case when $p=2$. 
This allows us to present the key ideas for a more general result.  
There are different equivalent formulations of Bloom's $ BMO _{\rho }$ space, two of which are  detailed in Section \ref{s:equiv}. 
These formulations are ideal for characterizing certain two weight inequalities for paraproducts in Section \ref{s:paraproducts}.   
Then, $ [b, H]$ is a linear combination of compositions of $ H$ with paraproducts, plus 
an error term, as detailed in \eqref{e:expand}.  The Hilbert transform is bounded on $ L^2$ of an $ A_2$ weight, thus, an upper bound for the commutator follows in Section \ref{s:Hilbert}. 
For the lower bound, a standard argument reveals yet another formulation of the $ BMO _{\rho }$ condition in 
\eqref{e:NecCon}.  
In a subsequent paper the authors will show how Bloom's result can be extended to all Calder\'on-Zygmund operators in arbitrary dimension and when $1<p<\infty$.

As the reader will see, there are four different equivalent definitions of Bloom's $ BMO _{\rho }$ space.  
It is hardly clear which is the best condition.   Also, the $ A_2$ condition will be appealed to repeatedly. 
For both reasons, we do not attempt to track the dependence on the $ A_2$ norms of the two weights.     
In particular $ A  \lesssim B  $ means that there is an absolute constant $ C$,  so that 
$ A  \leq C ([\lambda ] _{A_2} [\mu ] _{A_2}) ^{C} B $.  

\section{Equivalences for Bloom's BMO}
\label{s:equiv}

One of the interesting points, implicit in Bloom's work, is that the $ BMO _{\rho }$ space presents itself in different 
formulations at different points of the proof.   In this section, we make these alternate definitions precise, 
and do so in the dyadic setting.  Thus, $ \mathcal D$ denotes the standard dyadic grid on $ \mathbb R $, and for 
$ I\in \mathcal D$, the Haar function associated to $ I$ is 
\begin{equation*}
h _{I}  \equiv \lvert  I\rvert ^{-1/2}  ( - \mathbf 1_{I _{-}}   + \mathbf 1_{I _{+}})   
\end{equation*}
where $ I _{\pm} $ are the left and right dyadic children of $ I$.  

For weights $\mu,\lambda\in A_2$, define
\begin{equation}
\label{e:Bloom1}
\mathbf B_{2} [\mu ,\lambda]\equiv \sup _{ K\in\mathcal{D}} \mu^{-1} (K) ^{-1/2} 
\left\lVert \sum_{I \;:\; I\subset K}  \widehat b (I)  \langle \mu^{-1}  \rangle_I h_I 
\right\rVert _{L ^{2}(\lambda)}. 
\end{equation} 
Above, $\widehat b (I)  = \langle b, h_I \rangle $, and $\La\cdot,\cdot\Ra$ denotes the usual inner product in \textit{unweighted} $L^2(\R)$.  
Note that by the boundedness of the square function on $L^2(w)$, \cite{Witwer}, this can equivalently be characterized by:
\begin{equation}
\label{e:Bloom2}
\mathbf B_{2} [\mu ,\lambda] ^2 \equiv \sup _{ K\in\mathcal{D}} \frac 1 {\mu^{-1} (K)} \sum_{I\subset K} \widehat b(I)^2 \left\langle \mu^{-1}\right\rangle_I^2 \left\langle \lambda\right\rangle_I 
\end{equation}

\begin{prop}\label{p:B=B} For $\mu,\lambda\in A_2$ there holds 
\begin{equation} \label{e:B=B}
\mathbf B _{2}[\mu, \lambda]  \simeq \mathbf B _{2}[\lambda^{-1}, \mu^{-1} ] \simeq  \lVert b\rVert _{BMO _{\rho }}
\end{equation}
where $  \lVert b\rVert _{BMO _{\rho }}$ denotes the dyadic variant of the $ BMO _{\rho }$ space.  
\end{prop}


\begin{proof}
We prove the first equivalence in \eqref{e:B=B}.  
Fix a interval $ I_0$ for which we  verify that 
\begin{equation*}
\sum_{I \subset I_0}  \widehat b (I) ^2 \langle \lambda \rangle_I ^2 \left\langle \mu^{-1}\right\rangle_I \lesssim 
\mathbf B _{2} [\mu ,\lambda] ^2 \lambda (I_0). 
\end{equation*}
This will show that $ \mathbf B _{2}[\lambda^{-1}, \mu^{-1} ] \lesssim\mathbf B _{2}[\mu, \lambda] $, 
and by symmetry the reverse inequality holds.

Construct stopping intervals by taking $ \mathcal S$ to be the maximal subintervals  $I\subset  I_0$ such that 
\begin{equation*}
\langle \lambda \rangle_{I} > C \langle \lambda \rangle_{I_0} \quad \textup{or} \quad 
\langle \lambda \rangle_{I} < C ^{-1} \langle \lambda \rangle_{I_0}, 
\end{equation*}
or the same conditions hold for $ \mu^{-1} $.  By the $ A _{\infty } $ properties of $ \mu , \mu ^{-1}, \lambda  $  and $ \lambda ^{-1}$, for  $C= C_{\mu,\lambda}>1$ sufficiently large, there holds 
\begin{equation*}
\sum_{S\in \mathcal S} \lambda (S) < \tfrac 12 \lambda (I_0).  
\end{equation*}
The small constant in front implies that we can recurse inside these intervals, and so it remains to bound the 
sum over intervals `above' the stopping intervals.  

Let $ \mathcal I $ denote that $ I\subset I_0$ which are not contained in any stopping interval. Note that the 
$ \mathbf B _{2} [\mu ,\lambda] $ condition implies that   
\begin{equation} \label{e:B=>}
\sum_{I' \in \mathcal I}
\widehat  b (I') ^2 \lesssim  \mathbf B _{2} [\mu ,\lambda]  ^2 \frac {  \mu  ^{-1}(I_0)} { \langle \mu^{-1}  \rangle _{I_0} ^2  \langle \lambda \rangle_{I_0}} =\mathbf B _{2} [\mu ,\lambda]  ^2 \frac {  \lvert  I_0\rvert } { \langle \mu^{-1}  \rangle _{I_0}  \langle \lambda \rangle_{I_0}}  
\end{equation}
Therefore, 
\begin{align*}
\sum_{I  \in \mathcal I}  \widehat b (I) ^2 \langle \lambda \rangle_I ^2\left\langle \mu^{-1}\right\rangle_I 
& \lesssim 
\langle \lambda \rangle _{I_0} ^2 \langle \mu^{-1}  \rangle _{I_0}  \sum_{I  \in \mathcal I}  \widehat b (I)^2  
\\
& \lesssim  \mathbf B _{2} [\mu ,\lambda]^2  \lambda (I_0).  
\end{align*}
Hence we have that $\mathbf B_2[\lambda^{-1},\mu^{-1}]\lesssim \mathbf B_{2}[\mu,\lambda]$.  This argument is symmetric and so the result follows.

\bigskip 

We now show that 
$
\lVert b\rVert _{BMO _{\rho }} \lesssim \mathbf B _{2}[\mu ,\lambda] 
$, establishing first an intermediate result.  
Use the same stopping interval construction as in the previous argument.  Then, we have by Cauchy-Schwartz and \eqref{e:B=>},  
\begin{align*} 
\int _{I_0} \biggl[ 
\sum_{I\in \mathcal I} \frac {\widehat b (I) ^2 } {\lvert  I\rvert } \mathbf 1_{I}
\biggr] ^{1/2} \;dx 
& \lesssim  \mathbf B _{2} [\mu ,\lambda]    \frac {\lvert  I_0\rvert } { [\langle \mu^{-1}  \rangle _{I_0} \langle \lambda \rangle _{I_0} ] ^{1/2} } 
\\&
\lesssim \mathbf B _{2} [\mu ,\lambda] \frac {\lvert  I_0\rvert ^2  } 
{ [\mu^{-1} (I_0) \lambda (I_0)   ]}   & \textup{rewrite}
\\
&  \lesssim \mathbf B _{2} [\mu ,\lambda] \frac {\lvert  I_0\rvert ^2  }  {( \mu^{-1/2} \lambda ^{1/2} )(I_0)}  
& \textup{by H\"older's}
\\&\lesssim \mathbf B _{2} [\mu ,\lambda]  \rho (I_0)  & \textup{$ \rho = (\mu  /\lambda ) ^{1/2}  \in A_2$. }
\end{align*}
Here, we use the estimate \eqref{e:B=>}, then H\"older's inequality, to get to the product of 
the two $ A_2$ weights. The product is again an $ A_2$ weight, which is the last property used. 

It follows from this that we have proved a bound for an  $ L ^{1}$ BMO condition, namely 
\begin{equation} \label{e:1BMO}
\sup _{I_0} \frac 1 {\rho (I_0)}
\int _{I_0} \biggl[ 
\sum_{I \;:\; I\subset I_0} \frac {\widehat b (I) ^2 } {\lvert  I\rvert } \mathbf 1_{I}
\biggr] ^{1/2} \;dx  \lesssim \mathbf B _{2} [\mu ,\lambda]  . 
\end{equation}
Bloom's definition however includes a square inside the integral, see \eqref{e:bmo-rho}. 
To show that the condition above is the same as in \eqref{e:bmo-rho}, run another stopping condition, 
and again appeal to the fact that $ \rho \in A_2$.  

Let $ \mathcal S$ be the maximal  intervals  $S\subset  I_0$ such that 
\begin{equation*}
\sum_{I \;:\;  S \subset I\subset I_0} \frac {\widehat b (I) ^2 } {\lvert  I\rvert } \ge C 
\mathbf B _{2} [\mu ,\lambda]  ^2  \langle \rho  \rangle _{I_0}  
\end{equation*}
For $ C = C _{\rho }$ sufficiently large, there holds 
\begin{equation*}
\sum_{S \in \mathcal S} \rho (S) \leq \tfrac 12 \rho (I_0), 
\end{equation*}
and so we can recurse on these intervals.   Let $ \mathcal I$ be those intervals contained in $ I_0$ but not contained 
in any $ S\in \mathcal S$. There holds 
\begin{align*}
\int _{I_0}  \sum_{I \in \mathcal I} \frac {\widehat b (I) ^2 } {\lvert  I\rvert } \mathbf 1_{I}
  \;dx  
  & \lesssim 
  \mathbf B _{2} [\mu ,\lambda] ^2   \langle \rho  \rangle _{I_0}  \lvert  I_0\rvert \lesssim 
  \mathbf B _{2} [\mu ,\lambda]  ^2   \rho  (I_0).  
\end{align*}
That implies that $ \lVert b\rVert _{BMO _{\rho }} \lesssim \mathbf B _{2}[\mu ,\lambda] $.

\bigskip 

We show that $ \mathbf B _{2}[\mu ,\lambda]    \lesssim \lVert b\rVert _{BMO _{\rho }}$.  
Fix the interval $ I_0$ on which we will verify the $ \mathbf B _{2}[\mu ,\lambda] $ condition.  
We need stopping conditions, so let $ \mathcal S$ be the maximal dyadic intervals $ I\subset I_0$ such that 
one of three conditions is met: 
\begin{itemize}
\item[(1)] $ \langle  \mu^{-1}  \rangle _{I} > C \langle  \mu^{-1}  \rangle _{I_0}$, 
\item[(2)] $ \langle  \rho  \rangle _{I} > C \langle  \rho  \rangle _{I_0}$, or 
\item[(3)]
$
\sum_{I' \;:\; I \subset I'\subset I_0}  \widehat b (I') ^2 \lvert  I'\rvert ^{-1} 
 > [C_b   \langle  \rho  \rangle _{I_0} ] ^2 $. 
\end{itemize}
For $1\leq j\leq 3$, let $ \mathcal S_j$ be those intervals $ S\in \mathcal S$ which meet the  condition $ (j)$. For the first condition, there holds 
\begin{equation*}
\sum_{S\in \mathcal S_1} \mu^{-1} (S) \le \tfrac 14 \mu^{-1} (I_0), 
\end{equation*}
and so we can recurse on those intervals.  For the second condition, there holds 
\begin{equation*}
\sum_{S\in \mathcal S_2}  \lvert  S\rvert  \le  \epsilon _C   \lvert  I_0\rvert , 
\end{equation*}
by the $ A _{\infty }$ condition for $ \rho $.  Here $ \epsilon _C $ can be made arbitrarily small. 
The same condition holds for $ \mathcal S_3$, but this is just the usual John-Nirenberg estimate. 
The weight $ \mu^{-1}  $ is also $ A_ \infty $, so that for $ \epsilon _C $ sufficiently small, we see that 
\begin{equation*}
\sum_{S\in \mathcal S} \mu^{-1} (S) \le \tfrac 12 \mu^{-1} (I_0), 
\end{equation*}
and so we can recurse inside this collection. It remains to estimate the sum over $ I \subset I_0$ which 
are not contained in a interval $ S\in \mathcal S$. Calling this collection $ \mathcal I$, we have 
\begin{align*}
\int _{I_0} 
\sum_{I' \in \mathcal I}  \widehat b(I')^2  \langle \mu^{-1}  \rangle_I ^2  \frac {\mathbf 1_{I}(x)} {\lvert  I'\rvert } \; \lambda(x) dx
&\lesssim  \langle  \mu^{-1}  \rangle _{I_0} ^2 \langle \rho  \rangle_{I_0} ^2  \lambda (I_0) 
\\
& \lesssim\mu^{-1} (I_0) \frac {\mu^{-1}(I_0)    \mu (I_0)  \lambda^{-1} (I_0) \lambda (I_0) } {\lvert  I_0\rvert ^{4} }
\lesssim \mu^{-1} (I_0). 
\end{align*}
Here, we have just used the stopping conditions, then used the easy bound 
$ \rho (I_0) ^2 \le   \mu (I_0)  \lambda^{-1} (I_0) $, and finally 
appealed to the $ A_2$ conditions on $ \mu^{-1} $ and $ \lambda$.  
\end{proof}

\section{Two Weight Inequalities for Paraproduct Operators}
\label{s:paraproducts}

The `paraproduct' operator with symbol function $b$, and its dual,  are defined by 
\begin{align*}
\Pi_b&\equiv \sum_{I\in\mathcal{D}} \widehat{b}(I) h_I\otimes \frac{\mathsf{1}_I}{\left\vert I\right\vert}, 
\\ \textup{and} \qquad 
\Pi_b^{\ast}&\equiv \sum_{I\in\mathcal{D}} \widehat{b}(I) \frac{\mathsf{1}_I}{\left\vert I\right\vert}\otimes h_I. 
\end{align*}
Note that   $\Pi_b^{\ast}$ is the adjoint of the paraproduct on \textit{unweighted} $L^2(\mathbb{R})$. Using the identification $\left(L^2(w)\right)^* \equiv L^2(w^{-1})$, with pairing $\left<f,g\right>$ for all $f \in L^2(w)$ and $g \in L^2(w^{-1})$,  we can see that
	$$\text{The adjoint of } \Pi_b : L^2(\mu) \rightarrow L^2(\lambda) \text{ is } \Pi_b^* : L^2(\lambda^{-1}) \rightarrow L^2(\mu^{-1}); $$
	$$\text{The adjoint of } \Pi^*_b : L^2(\mu) \rightarrow L^2(\lambda) \text{ is } \Pi_b : L^2(\lambda^{-1}) \rightarrow L^2(\mu^{-1}). $$

The characterization of the boundedness of these operators between weighted spaces $L^2(\mu)$ and $L^2(\lambda)$ is as follows.  

\begin{thm}
\label{t:paraproduct}
Let $\mu,\lambda\in A_2$.  Suppose that $\mathbf B_2[\mu,\lambda]$ and $\mathbf B_2[\lambda^{-1},\mu^{-1}]$ finite.
Then we have
\begin{eqnarray}
\label{e:Pib1} & & \left\Vert \Pi_b:L^2(\mu)\to L^2(\lambda)\right\Vert  = \left\Vert \Pi^*_b:L^2(\lambda^{-1})\to L^2(\mu^{-1})\right\Vert  \simeq    \mathbf B_2[\mu,\lambda] \\
\label{e:Pib*1} & &  \left\Vert \Pi_b^{\ast}:L^2(\mu)\to L^2(\lambda)\right\Vert = \left\Vert \Pi_b:L^2(\lambda^{-1})\to L^2(\mu^{-1})\right\Vert  \simeq   \mathbf B_2[\lambda^{-1},\mu^{-1}].
\end{eqnarray}
\end{thm}

\begin{proof}[Proof of Sufficiency in Theorem \ref{t:paraproduct}]
Before the proof, recall that for any weight $w\in A_2$ we have:
\begin{equation}\label{e:PPott}
\sum_{I \in \mathcal{D}} |\widehat{f}(I)|^2 \frac{1}{\left<w\right>_I} \lesssim [w]_{A_2} \|f\|^2_{L^2(w^{-1})},
\end{equation}
which can be found in \cite{PetermichlPott}.  

Note that since  $\mathbf B_2[\mu,\lambda]$ and $\mathbf B_2[\lambda^{-1},\mu^{-1}]$ are finite:
\begin{eqnarray}
\sum_{I\subset J} \widehat b(I)^2\left\langle \lambda\right\rangle_{I} \left\langle \mu^{-1} \right\rangle_{I}^2  & \leq & \mathbf B_2[\mu,\lambda]^2 \mu^{-1}(J)\quad\forall J\in\mathcal{D}\label{e:CET1nec}\\
\sum_{I\subset J} \widehat b(I)^2 \left\langle \mu^{-1}\right\rangle_I \left\langle \lambda\right\rangle_I^2 & \leq & \mathbf B_2[\lambda^{-1},\mu^{-1}]^2\lambda(J) \quad\forall J\in\mathcal{D}\label{e:CET2nec}.
\end{eqnarray}
These conditions will imply the certain measures are Carleson, and so we can then appeal to the Carleson Embedding Theorem to control terms directly.

We proceed by duality to analyze the operator $\Pi_b$.  Note that for $f\in L^2(\mu)$ and $g\in L^2(\lambda)$ we have
\begin{eqnarray*}
\left\vert \left\langle \Pi_b f,g\right\rangle_{L^2(\lambda)} \right\vert & \leq  & \sum_{I\in\mathcal{D}}  \left\vert\widehat{b}(I) \left\langle f \right\rangle_I \left\langle g, h_I\right\rangle_{L^2(\lambda)}\right\vert = \sum_{I\in\mathcal{D}}  \left\vert \widehat{b}(I) \left\langle \mu^{-1}\right\rangle_{I} \mathbb{E}_{I}^{\mu^{-1}}(f\mu) \left\langle g, h_I\right\rangle_{L^2(\lambda)}\right\vert\\
& = & \sum_{I\in\mathcal{D}} \left\vert\widehat{b}(I) \left\langle \mu^{-1}\right\rangle_{I} \left\langle\lambda\right\rangle_{I}^{\frac{1}{2}} \left\langle\lambda\right\rangle_{I}^{-\frac{1}{2}}\mathbb{E}_{I}^{\mu^{-1}}(f\mu) \left\langle g, h_I\right\rangle_{L^2(\lambda)}\right\vert\\
& \leq & \left(\sum_{I\in\mathcal{D}}  \hat{b}(I)^2 \left\langle \mu^{-1}\right\rangle_{I}^2\left\langle \lambda\right\rangle_{I} \mathbb{E}_{I}^{\mu^{-1}}(f\mu)^2 \times \sum_{I\in\mathcal{D}} \frac{\left\vert \widehat{g\lambda}(I)\right\vert^2}{\left\langle \lambda\right\rangle_I}\right)^{\frac{1}{2}}\\
& \leq & 
\mathbf B_2[\mu,\lambda]\left\Vert \mu f\right\Vert_{L^2(\mu^{-1})} \left\Vert g\lambda\right\Vert_{L^2(\lambda^{-1})}\\
& = & 
\mathbf B_2[\mu,\lambda]\left\Vert f\right\Vert_{L^2(\mu)} \left\Vert g\right\Vert_{L^2(\lambda)}.
\end{eqnarray*}
Here, we have used the Carleson Embedding Theorem to control the term with the averages on $f$, which is applicable by \eqref{e:CET1nec}, and we have used \eqref{e:PPott} to handle the other term.  The claimed estimate, \eqref{e:Pib1}, on the norm of $\Pi_b:L^2(\mu)\to L^2(\lambda)$ follows.

We next turn to controlling $\Pi_b^*$ and again resort to duality to estimate the norm.  Indeed, we have
\begin{eqnarray*}
\left\vert \left\langle \Pi_b^* f,g\right\rangle_{L^2(\lambda)} \right\vert & \leq  & \sum_{I\in\mathcal{D}}  \left\vert \widehat{b}(I) \left\langle g\lambda \right\rangle_I \left\langle f, h_I\right\rangle_{L^2}\right\vert =  \sum_{I\in\mathcal{D}}  \left\vert\widehat{b}(I) \left\langle \mu^{-1}\right\rangle_I^{\frac{1}{2}} \left\langle \lambda\right\rangle_{I} \mathbb{E}_I^{\lambda}\left\langle g \right\rangle_I \frac{\widehat{f}(I)}{\left\langle \mu^{-1}\right\rangle_I^{\frac{1}{2}}}\right\vert\\
& \leq & \left(\sum_{I\in\mathcal{D}}  \hat{b}(I)^2 \left\langle \mu^{-1}\right\rangle_{I}\left\langle \lambda\right\rangle_{I}^{2} \mathbb{E}_{I}^{\lambda}(g)^2 
\times 
\sum_{I\in\mathcal{D}} \frac{\widehat{f}(I)^2}{\left\langle \mu^{-1}\right\rangle_I}\right)^{\frac{1}{2}}\\
& \leq &
\mathbf B_2[\lambda^{-1},\mu^{-1}]\left\Vert f\right\Vert_{L^2(\mu)} \left\Vert g\lambda\right\Vert_{L^2(\lambda)},
\end{eqnarray*}
with the inequality following by the Carleson Embedding Theorem since we are imposing condition \eqref{e:CET2nec} and also using \eqref{e:PPott}.  Combining all these estimates, we see that \eqref{e:Pib*1} holds.
\end{proof}

\begin{proof}[Proof of Necessity in Theorem \ref{t:paraproduct}]
Fix an interval $ I$, and choose $ f=\mu^{-1}\mathbf 1_I$.  Then we have $ \lVert f\rVert _{L ^{2} (\mu )} = \mu^{-1} (I) ^{1/2} $.   Then we have:  
\begin{align*}
\left\lVert \sum_{I \;:\; I\subset I} 
 \widehat b (I)  \langle \mu^{-1}  \rangle_I h_I 
\right\rVert _{L ^{2}(\lambda)}
\le \lVert \Pi _{b} f \rVert _{L ^{2} (\lambda)} \leq \left\Vert \Pi_b:L^2(\mu)\to L^2(\lambda)\right\Vert \mu^{-1} (I) ^{1/2}, 
\end{align*}
with the last inequality following from the assumed norm boundedness of the paraproduct.  Hence, we have that:
	\begin{equation} \label{E:PibNec}
	\mathbf{B}_2[\mu,\lambda]\leq \left\Vert \Pi_b:L^2(\mu)\to L^2(\lambda)\right\Vert.
	\end{equation}

In light of our previous discussion about adjoints, proving the necessity for $\Pi_b$ will address $\Pi_b^*$ as well.  Since if $\Pi^*_b : L^2(\mu) \rightarrow L^2(\lambda)$ is bounded, then $\Pi_b : L^2(\lambda^{-1}) \rightarrow L^2(\mu^{-1})$ is bounded, with the same operator norm. From \eqref{E:PibNec}, we have then
	$$B_2[\lambda^{-1},\mu^{-1}] \leq \left\Vert \Pi_b:L^2(\lambda^{-1})\to L^2(\mu^{-1})\right\Vert = 
	\left\Vert \Pi^*_b:L^2(\mu)\to L^2(\lambda)\right\Vert.$$

\end{proof}

\section{Proof of Bloom's Theorem, $ p=2$} 
\label{s:Hilbert}

%


For the sufficiency, we use Petermichl's beautiful observation that the Hilbert transform can be recovered through an appropriate average of Haar shifts, \cite{MR1756958}. On  the dyadic lattice $\mathcal{D}$ with Haar basis $\{h_I\}_{I\in\mathcal{D}}$, we define $\Sh h_I=\frac{1}{\sqrt{2}}(h_{I_-}-h_{I_+})$, which is Petermichl's Haar shift operator.  Then, the Hilbert transform is an average of shift operators, with the average performed over the class of 
all dyadic grids.  In particular, to prove norm inequalities for the Hilbert transform, it suffices to prove them 
for the Haar shift operator, which has proven to be a powerful proof technique.

The commutator with the Haar shift operator has an explicit expansion in terms of the paraproducts and $\Sh $, see \cite{MR1756958} for this decomposition,
 \begin{equation}
\label{e:expand}
[b,\Sh]f=\Sh(\Pi_bf)-\Pi_b(\Sh f)+\Sh (\Pi_b^* f)-\Pi_b^*(\Sh f)+\Pi_{\Sh f} b-\Sh(\Pi_f b).
\end{equation}
For any $w\in A_2$, $\left\Vert \Sh:L^2(w)\to L^2(w)\right\Vert\lesssim [w]_{A_2}$, \cite{MR2354322}.  
Thus, for the first four terms above, we merely have to control the paraproduct term. But this is done in   Theorem \ref{t:paraproduct}. 
Thus, we see that:
$$
\left\Vert [b,\Sh]f\right\Vert_{L^2(\lambda)}\lesssim \left(\mathbf B_2[\lambda^{-1},\mu^{-1}]+\mathbf B_2[\mu,\lambda]\right)\left\Vert f\right\Vert_{L^2(\mu)}+\left\Vert \Pi_{\Sh f} b-\Sh(\Pi_f b)\right\Vert_{L^2(\lambda)}.
$$
The last two terms in \eqref{e:expand} have more cancellation than the other four terms. By direct calculation, 
$$
\Pi_{\Sh f} b-\Sh(\Pi_f b)=\sum_{I\in\mathcal{D}} \frac{\widehat{b}(I)}{\left\vert I\right\vert^{\frac{1}{2}}} \widehat{f}(I) (h_{I_+}-h_{I_-}).
$$
We then show that:
\begin{eqnarray*}
\left\Vert \Pi_{\Sh f}  b-\Sh(\Pi_f b)\right\Vert_{L^2(\lambda)}^2 & \lesssim & 
\left\Vert S(\Pi_{\Sh f}  b-\Sh(\Pi_f b))\right\Vert_{L^2(\lambda)}^2\\
& \lesssim &
\sum_{I\in\mathcal{D}} \frac{\widehat b(I)^2}{\left\vert I\right\vert} \left\langle \lambda\right\rangle_I \hat{f}(I)^2\\
& = & 
\sum_{I\in\mathcal{D}} \frac{\widehat b(I)^2}{\left\vert I\right\vert} \left\langle \lambda\right\rangle_I \left\langle \mu^{-1}\right\rangle_I \frac{\hat{f}(I)^2}{\left\langle \mu^{-1}\right\rangle_I}.
\end{eqnarray*}
Note that \eqref{e:CET1nec} and \eqref{e:CET2nec} imply that:
$$
\sup_{I\in\mathcal{D}} \frac{ \hat{b}(I)^2\left\langle\lambda\right\rangle_I\left\langle\mu^{-1}\right\rangle_I}{\left\vert I\right\vert}\leq \mathbf B_2[\lambda^{-1},\mu^{-1}]\mathbf B_2[\mu,\lambda].
$$
And, so we then have:
\begin{eqnarray*}
\left\Vert \Pi_{\Sh f} b-\Sh(\Pi_f b)\right\Vert_{L^2(\lambda)}^2 & \lesssim &
\mathbf B_2[\lambda^{-1},\mu^{-1}]\mathbf B_2[\mu,\lambda]\sum_{I\in\mathcal{D}} \frac{ \hat{f}(I)^2}{\left\langle \mu^{-1}\right\rangle_I}\\
& \lesssim & 
\mathbf B_2[\lambda^{-1},\mu^{-1}]\mathbf B_2[\mu,\lambda] \left\Vert f\right\Vert_{L^2(\mu)}^2.
\end{eqnarray*}

We now turn to the converse result.  Assume that there holds 
\begin{equation*}
\lVert   [ b , H]  : L ^{2} (\mu ) \mapsto L ^{2} (\lambda)\rVert <\infty.
\end{equation*}
Using an argument of  Coifman--Rochberg--Weiss, \cite{CRW}, we 
derive a new necessary condition, and show that it dominates Bloom's condition. 

Let $ I$ be an interval centered at the origin, and set $ S_I = \mathbf 1_{I}\, \textup{sgn} (b - \langle b \rangle_I)$.  We have 
\begin{align*}
\lvert  I\rvert \cdot \bigl\lvert  (b& - \langle b \rangle_I) \mathbf 1_{I}\bigr\rvert 
\\
&= 
\int _{I} \frac {b (x) - b (y)} { x-y }(x-y)S_I (x) \mathbf 1_{I} (y) \; dy 
\\
&= x S_I (x)  \bigl\{[b, H ]  (\mathbf 1_{I} (y)) \bigr\} (x)-S_I (x)  \bigl\{[b, H ]   (y\mathbf 1_{I} (y)) \bigr\} (x).
\end{align*}
The assumed norm inequality 
then implies that 
\begin{align}
\lvert  I\rvert ^2 \int _{I} \lvert  b(x) - \langle b \rangle _I \rvert ^2 \; \lambda(x) dx 
\lesssim  \lvert  I\rvert ^{2} \mu (I) \lVert  [ b, H] : L ^{2} (\mu) \mapsto L ^{2} (\lambda)\rVert^2.   
\end{align}
The assumption that $ I$ is centered at the origin then allows us to dominate $ \lvert  x\rvert \lesssim \left\vert I\right\vert$.  The centering is a harmless assumption, and so we deduce the necessary condition 
\begin{equation} \label{e:necc1}
\sup _{I}   \frac 1 {\mu (I)}    \int _{I} \lvert  b(x) - \langle b \rangle _I \rvert ^2 \; \lambda(x) dx  \lesssim \lVert  [ b, H] : L ^{2} (\mu) \mapsto L ^{2} (\lambda)\rVert^2.  
\end{equation}
But $ \mu \in A_2$, which implies that: 
\begin{equation*}
1\leq \frac  {\mu  (I)\mu^{-1} (I)}    {\lvert  I\rvert  ^2 } \leq \left[ \mu \right]_{A_2}\quad\forall I\in\mathcal{D}.  
\end{equation*}
Hence, we see that 
\begin{equation}
\label{e:NecCon}
\sup _{I}   \frac  {\mu^{-1}  (I)}  {\lvert  I\rvert  ^2 } 
 \int _{I} \lvert  b(x) - \langle b \rangle _I \rvert ^2 \; \lambda(x) dx  \lesssim 
 \lVert  [ b, H] : L ^{2} (\mu) \mapsto L ^{2} (\lambda)\rVert^2.  
\end{equation}

\smallskip 
We show  that \eqref{e:NecCon} implies  that $\mathbf B _{2} [\mu ,\lambda]$ is finite.  As we have already shown that this is equivalent to the Bloom condition, it implies that the boundedness of the commutator implies that $b$ belongs to the Bloom BMO space.  Recall that, 
\begin{equation*}
\mathbf B_{2} [\mu ,\lambda]:= 
\sup _{K\in\mathcal{D}} \mu^{-1} (K) ^{-1/2} 
\left\lVert \sum_{I \;:\; I\subset K} 
 \widehat b (I)  \langle \mu^{-1}  \rangle_I h_I 
\right\rVert _{L ^{2}(\lambda)}. 
\end{equation*}

Fix a interval $ I_0$ on which we need to verify the $\mathbf B_2[\mu,\lambda]$ condition. 
Let $ \mathcal S$ be the maximal stopping intervals $ S\subset I_0$ so that $ \langle \mu^{-1}  \rangle_S \ge 4 \langle \mu^{-1}  \rangle _{I_0}$.  By the $ A _{\infty }$ property of $ \mu^{-1} $, it suffices to restrict the sum above to $ I\subset I_0$ 
with $ I $ not contained in any stopping interval.  But then we have
\begin{align*}
\sum_{I \in \mathcal I} \bigl\lvert \widehat b (I)\bigr\rvert ^2 \langle \mu^{-1}  \rangle_I^2
\left\langle \lambda\right\rangle_I & \lesssim 
\langle \mu^{-1}  \rangle_{I_0} ^2 
\sum_{I \;:\; I\subset I_0}\left\vert \widehat b (I)  \right\vert^2  \frac {\lambda (I)} {\lvert  I\rvert }
\\
& \lesssim 
\langle \mu^{-1}  \rangle_{I_0} ^2  \int _{I_0} \bigl\lvert b(x) - \langle b \rangle_ {I_0}\bigr\rvert ^2 \; \lambda(x) dx 
\\
& \lesssim \langle \mu^{-1}  \rangle_{I_0} ^2   \frac { \lvert  I_0\rvert ^2  } {\mu^{-1} (I_0)} \sup _{I}   \frac  {\mu^{-1}  (I)}  {\lvert  I\rvert  ^2 } 
 \int _{I} \lvert  b - \langle b \rangle _I \rvert ^2 \; \lambda(x) dx\\
 & = \mu^{-1} (I_0)\sup _{I}   \frac  {\mu^{-1}  (I)}  {\lvert  I\rvert  ^2 } 
 \int _{I} \lvert  b - \langle b \rangle _I \rvert ^2 \; \lambda(x) dx. 
\end{align*}
Therefore, $ \textbf B _{2} [ \mu ,\lambda] $ is bounded by $\lVert  [ b, H] : L ^{2} (\mu) \mapsto L ^{2} (\lambda)\rVert$ and so $b\in BMO_{\rho}$, and the proof is complete.  





\begin{bibdiv}
\begin{biblist}

\normalsize

\bib{Bloom}{article}{
   author={Bloom, S.},
   title={A commutator theorem and weighted BMO},
   journal={Trans. Amer. Math. Soc.},
   volume={292},
   date={1985},
   number={1},
   pages={103--122}
}

\bib{CRW}{article}{
   author={Coifman, R. R.},
   author={Rochberg, R.},
   author={Weiss, Guido},
   title={Factorization theorems for Hardy spaces in several variables},
   journal={Ann. of Math. (2)},
   volume={103},
   date={1976},
   number={3},
   pages={611--635}
}

\bib{PetermichlPott}{article}{
   author={Petermichl, S.},
   author={Pott, S.},
   title={An estimate for weighted Hilbert transform via square functions},
   journal={Trans. Amer. Math. Soc.},
   volume={354},
   date={2002},
   number={4},
   pages={1699--1703 (electronic)}
}

\bib{MR2354322}{article}{
   author={Petermichl, S.},
   title={The sharp bound for the Hilbert transform on weighted Lebesgue
   spaces in terms of the classical $A_p$ characteristic},
   journal={Amer. J. Math.},
   volume={129},
   date={2007},
   number={5},
   pages={1355--1375}
}

\bib{MR1756958}{article}{
   author={Petermichl, S.},
   title={Dyadic shifts and a logarithmic estimate for Hankel operators with
   matrix symbol},
   journal={C. R. Acad. Sci. Paris S\'er. I Math.},
   volume={330},
   date={2000},
   number={6},
   pages={455--460}
}

\bib{Witwer}{article}{
   author={Wittwer, Janine},
   title={A sharp estimate on the norm of the martingale transform},
   journal={Math. Res. Lett.},
   volume={7},
   date={2000},
   number={1},
   pages={1--12}
}

\end{biblist}
\end{bibdiv}
\end{document}